\documentclass{article}
\usepackage{tikzfig}
\tikzstyle{every picture}=[baseline=-0.25em]
\tikzstyle{dotpic}=[scale=0.6]
\tikzstyle{diredges}=[every to/.style={diredge}]
\tikzstyle{dot graph}=[shorten <=-0.1mm,shorten >=-0.1mm,scale=0.6]
\tikzstyle{digraph}=[-latex]
\tikzstyle{plot point}=[circle,fill=black,minimum width=2mm,inner sep=0]
\tikzstyle{string graph}=[scale=0.6]
\tikzstyle{sg diredge}=[-stealth]
\tikzstyle{rewrite edge}=[-open triangle 45]
\tikzstyle{sg bold diredge}=[-stealth,thick,shorten >=-1pt]
\tikzstyle{sg vertex}=[circle,minimum width=2.2mm,fill=white,draw=black,inner sep=0mm]
\tikzstyle{labelled sg vertex}=[circle,minimum width=7mm,fill=white,draw=black,inner sep=0mm]
\tikzstyle{sg grey vertex}=[sg vertex,fill=gray!30!white]
\tikzstyle{sg black vertex}=[sg vertex,fill=black]
\tikzstyle{sg bold vertex}=[circle,minimum width=2.2mm,fill=white,draw=black,very thick,inner sep=0mm]
\tikzstyle{sg wire vertex}=[circle,minimum width=1mm,fill=black,inner sep=0mm]

\tikzstyle{tick vertex}=[rectangle,fill=black,minimum height=1mm,minimum width=2.5mm,inner sep=0mm]

\tikzstyle{braceedge}=[decorate,decoration={brace,amplitude=2mm,raise=-1mm}]
\tikzstyle{small braceedge}=[decorate,decoration={brace,amplitude=1mm,raise=-1mm}]
\tikzstyle{left hook arrow}=[left hook-latex]
\tikzstyle{right hook arrow}=[right hook-latex]

\tikzstyle{dot}=[inner sep=0.7mm,minimum width=0pt,minimum height=0pt,fill=black,draw=black,shape=circle]
\tikzstyle{white dot}=[dot,fill=white]
\tikzstyle{alt white dot}=[white dot,label={[xshift=2.9mm,yshift=-0.1mm]left:$\cdot$}]
\tikzstyle{gray dot}=[dot,fill=gray!50]
\tikzstyle{box vertex}=[draw=black,rectangle]
\tikzstyle{whitebg}=[fill=white,inner sep=2pt]
\tikzstyle{graph state vertex}=[sg vertex,fill=black]

\tikzstyle{wide point}=[fill=white,draw=black,shape=isosceles triangle,shape border rotate=90,isosceles triangle stretches=true,inner sep=1pt,minimum width=1.5cm,minimum height=5mm]
\tikzstyle{wide copoint}=[fill=white,draw=black,shape=isosceles triangle,shape border rotate=-90,isosceles triangle stretches=true,inner sep=1pt,minimum width=1.5cm,minimum height=5mm]
\tikzstyle{symm}=[ultra thick,shorten <=-1mm,shorten >=-1mm]

\tikzstyle{square box}=[rectangle,fill=white,draw=black,minimum height=6mm,minimum width=6mm]
\tikzstyle{square gray box}=[rectangle,fill=gray!30,draw=black,minimum height=6mm,minimum width=6mm]
\tikzstyle{point}=[regular polygon,regular polygon sides=3,draw=black,scale=0.75,inner sep=-0.5pt,minimum width=7mm,fill=white]
\tikzstyle{copoint}=[point,regular polygon rotate=180,fill=white]
\tikzstyle{gray point}=[point,fill=gray!40!white]
\tikzstyle{gray copoint}=[copoint,fill=gray!40!white]

\tikzstyle{open graph}=[baseline=-0.25em]
\tikzstyle{greybg}=[background rectangle/.style={fill=black!5,draw=black!30,rounded corners=1ex}, show background rectangle]

\tikzstyle{edge point}=[circle,minimum width=1mm,fill=black,inner sep=0mm]
\tikzstyle{vertex point}=[circle,minimum width=2.2mm,fill=white,draw=black,inner sep=0mm]
\tikzstyle{gray vertex point}=[circle,minimum width=2.2mm,fill=gray!30!white,draw=black,inner sep=0mm]
\tikzstyle{edge label}=[inner sep=2pt, font=\small]
\tikzstyle{on edge label}=[fill=white, font=\footnotesize, inner sep=1 pt]

\newcommand{\edgearrow}{{\arrow[black]{>}}}
\newcommand{\edgetick}{{\arrow[black,scale=0.7,very thick]{|}}}
\tikzstyle{diredge}=[postaction=decorate,decoration={markings, mark=at position 0.55 with \edgearrow}]
\tikzstyle{medium diredge}=[postaction=decorate,decoration={markings, mark=at position 0.75 with \edgearrow}]
\tikzstyle{short diredge}=[->]
\tikzstyle{halfedge}=[-)]
\tikzstyle{other halfedge}=[(-]
\tikzstyle{freeedge}=[(-)]
\tikzstyle{white edge}=[line width=5pt,white]
\tikzstyle{tick}=[postaction=decorate,decoration={markings, mark=at position 0.5 with \edgetick}]
\tikzstyle{small map edge}=[|-latex, gray!60!blue, shorten <=0.9mm, shorten >=0.5mm]
\tikzstyle{thick dashed edge}=[very thick,dashed,gray!40]
\tikzstyle{dashed edge}=[densely dotted,thick]
\tikzstyle{map edge}=[|-latex,very thick, gray!40, shorten <=1mm, shorten >=0.5mm]
\tikzstyle{tickedge}=[postaction=decorate,
  decoration={markings, mark=at position 0.5 with \edgetick}]
\tikzstyle{dirtickedge}=[postaction=decorate,
  decoration={markings, mark=at position 0.5 with \edgetick},
  decoration={markings, mark=at position 0.85 with \edgearrow}]
\tikzstyle{dirdoubletickedge}=[postaction=decorate,
  decoration={markings, mark=at position 0.4 with \edgetick},
  decoration={markings, mark=at position 0.6 with \edgetick},
  decoration={markings, mark=at position 0.85 with \edgearrow}]

\tikzstyle{arrs}=[-latex,font=\small,auto]
\tikzstyle{arrow plain}=[arrs]
\tikzstyle{arrow dashed}=[dashed,arrs]
\tikzstyle{arrow bold}=[very thick,arrs]
\tikzstyle{arrow hide}=[draw=white!0,-]
\tikzstyle{arrow reverse}=[latex-]
\tikzstyle{cdnode}=[]

\tikzstyle{cnot}=[fill=white,shape=circle,inner sep=-1.4pt]
\tikzstyle{bang box}=[draw=black,dashed,minimum height=12mm,minimum width=12mm,fill=gray!20]
\tikzstyle{wire label}=[font=\footnotesize, auto]

\usepackage{amsmath}
\usepackage{amssymb}
\usepackage{stmaryrd}
\usepackage{amsthm}
\usepackage{xspace}
\usepackage{hyperref}
\usepackage{graphicx}
\graphicspath{{figures/}}

\usepackage{algpseudocode}
\usepackage{algorithm}

\newtheorem{theorem}{Theorem}[section]
\newtheorem*{theorem*}{Theorem}

\newtheorem{lemma}[theorem]{Lemma}
\newtheorem{corollary}[theorem]{Corollary}
\theoremstyle{definition}
\theoremstyle{definition}
\theoremstyle{definition}\newtheorem{definition}[theorem]{Definition}
\theoremstyle{definition}
\theoremstyle{definition}
\theoremstyle{definition}
\theoremstyle{definition}
\theoremstyle{definition}
\theoremstyle{definition}
\theoremstyle{definition}\newtheorem{notation}[theorem]{Notation}

\newcommand{\atsFree}{\ensuremath{\textbf{Free}(\mathcal A)}\xspace}
\newcommand{\catFree}{\ensuremath{\mathbb C[\hspace{0.07em}\atsFree]}\xspace}
\newcommand{\catSig}{\ensuremath{\textbf{Sig}(\mathcal A)}\xspace}

\newcommand{\catTSMC}{\ensuremath{\textrm{\bf TSMC}}\xspace}

\newcommand{\Tr}{\ensuremath{\textrm{Tr}\xspace}}

\newcommand{\catMonSig}{\ensuremath{\textrm{\bf MonSig}}\xspace}

\newcommand{\dom}{\ensuremath{\textrm{\rm dom}}}
\newcommand{\cod}{\ensuremath{\textrm{\rm cod}}}

\tikzstyle{cdiag}=[matrix of math nodes, row sep=2em, column sep=2em, text height=1.5ex, text depth=0.25ex,inner sep=0.5em]
\tikzstyle{arrow above}=[transform canvas={yshift=0.5ex}]
\tikzstyle{arrow below}=[transform canvas={yshift=-0.5ex}]

\def\bR{\begin{color}{red}} 
\def\bB{\begin{color}{blue}}
\def\bM{\begin{color}{magenta}}
\def\bC{\begin{color}{cyan}}
\def\bW{\begin{color}{white}}
\def\bBl{\begin{color}{black}} 
\def\bG{\begin{color}{green}}
\def\bY{\begin{color}{yellow}}
\def\e{\end{color}}

\title{Abstract Tensor Systems as Monoidal Categories}
\author{Aleks Kissinger\\
\\
\textit{Dedicated to Joachim Lambek on the occasion of his 90th birthday}}

\newcommand{\labelsets}{\ensuremath{\mathcal P_f(\mathcal L)}\xspace}

\newcommand{\wx}{\ensuremath{\mathbf x}\xspace}
\newcommand{\wy}{\ensuremath{\mathbf y}\xspace}
\newcommand{\wz}{\ensuremath{\mathbf z}\xspace}

\newcommand{\vi}{\ensuremath{\vec i}\xspace}
\newcommand{\vj}{\ensuremath{\vec j}\xspace}

\newcommand{\vu}{\ensuremath{\vec u}\xspace}
\newcommand{\vv}{\ensuremath{\vec v}\xspace}
\newcommand{\vw}{\ensuremath{\vec w}\xspace}
\newcommand{\vx}{\ensuremath{\vec x}\xspace}
\newcommand{\vy}{\ensuremath{\vec y}\xspace}
\newcommand{\vz}{\ensuremath{\vec z}\xspace}
\newcommand{\K}{\ensuremath{\mathcal K}\xspace}

\newcommand{\ii}{{\,(0)}}
\newcommand{\oo}{{\,(1)}}

\newcommand{\CS}{\ensuremath{\mathbb C[\mathcal S]}\xspace}

\begin{document}

\maketitle

\begin{abstract}
  The primary contribution of this paper is to give a formal, categorical treatment to Penrose's abstract tensor notation, in the context of traced symmetric monoidal categories. To do so, we introduce a typed, sum-free version of an abstract tensor system and demonstrate the construction of its associated category. We then show that the associated category of the free abstract tensor system is in fact the free traced symmetric monoidal category on a monoidal signature. A notable consequence of this result is a simple proof for the soundness and completeness of the diagrammatic language for traced symmetric monoidal categories.
\end{abstract}

% This paper introduces defines three equivalent languages for describing morphisms in traced symmetric monoidal categories. The first is an algebraic, tensor-style language (essentially) due to Penrose. The second is the language of topological string diagrams formalised by Joyal and Street. The third language, due to the author, is graph-theoretic, and conducive to automation. The primary contribution of this paper is an exposition of these languages with minimal theoretical overhead and emphasis on their expressive capabilities, along with a proof of their mutual equivalence. A notable consequence of this equivalence is the proof of a long-standing ``folk theorem'' concerning the soundness of the graphical language for symmetric monoidal categories.

\section{Introduction}

This paper formalises the connection between monoidal categories and the abstract index notation developed by Penrose in the 1970s, which has been used by physicists directly, and category theorists implicitly, via the diagrammatic languages for traced symmetric monoidal and compact closed categories. This connection is given as a representation theorem for the free traced symmetric monoidal category as a syntactically-defined strict monoidal category whose morphisms are equivalence classes of certain kinds of terms called \textit{Einstein expressions}. Representation theorems of this kind form a rich history of coherence results for monoidal categories originating in the 1960s~\cite{MacLaneCoh,EpsteinCoh}. Lambek's contribution~\cite{LambekD1,LambekD2} plays an essential role in this history, providing some of the earliest examples of syntactically-constructed free categories and most of the key ingredients in Kelly and Mac Lane's proof of the coherence theorem for closed monoidal categories~\cite{Kelly1971}. Recently, Lambek has again become interested in the role of compact closed categories (a close relative of traced symmetric monoidal categories) in linguistics and physics, both contributing~\cite{LambekPhysics} and inspiring~\cite{LambekVLambek,CoeckeLinguistics2010} ground-breaking new work there. The present work also aims to build a bridge between monoidal categories and theoretical physics, by formalising the use of a familiar language in physics within a categorical context. For these reasons, the author would like to dedicate this exposition to Lambek, on the occasion of his 90th birthday.

Tensors are a fundamental mathematical tool used in many areas of physics and mathematics with notable applications in differential geometry, relativity theory, and high-energy physics. A (concrete) \textit{tensor} is an indexed set of numbers in some field (any field will do, but we'll use $\mathbb C$).
\begin{equation}\label{eq:concrete-tensor}
  \{ \psi_{i_1,\ldots,i_m}^{j_1,\ldots,j_n} \in \mathbb C \}_{i_k, j_k \in \{1,\ldots,D\}}
\end{equation}

One should think of the lower indices as \textit{inputs} and the upper indices as outputs. Notable special cases are column vectors $v^j$, row vectors $\xi_i$, and matrices $M_i^j$. Tensors can be combined via the \textit{tensor product} and \textit{contraction}. The product of two tensors is defined as a new tensor whose elements are defined point-wise as products.
\[ (\psi\phi)_{i,i'}^{j,j'} := \psi_i^j \phi_{i'}^{j'} \]

Contraction is a procedure by which a lower index is ``plugged into'' an upper index by summing them together.
\[ \theta_i^j := \sum_{k = 1}^D \psi_i^k \phi_k^j \]

Special cases of contraction are matrix composition, application of a matrix to a vector, and the trace of a matrix. It is customary to employ the \textit{Einstein summation convention}, whereby repeated indices are assumed to be summed over.
\[ \psi_{i,j}^k \phi_{k,l}^m \xi_m^i :=
     \sum_{i,k,m} \psi_{i,j}^k \phi_{k,l}^m \xi_m^i \]

In other words, indices occurring only once are considered \textit{free} in the tensor expression (and can be used in further contractions), whereas repeated indices are implicitly \textit{bound} by the sum.

Abstract tensor notation was defined by Penrose in 1971~\cite{Penrose1971} to give an elegant way to describe various types of multi-linear maps without the encumbrance of fixing bases. It allows one to reason about much more general processes with many inputs and outputs \textit{as if} they were just tensors. In that paper, he actually introduced two notations. He introduced a term-based notation, where the objects of interest are \textit{abstract Einstein expressions}, and an equivalent (more interesting) notation that is diagrammatic. There, tensors are represented as \textit{string diagrams}. This diagrammatic notation, as an elegant way of expressing tensor expressions, has appeared in popular science~\cite{PenroseRoad}, theoretical physics~\cite{PenroseSpinors}, representation theory~\cite{Birdtracks}, and (in its more abstract form) foundations of physics~\cite{HardyCircuits}.

Twenty years after Penrose's original paper, Joyal and Street formalised string diagrams as topological objects and showed that certain kinds of these diagrams can be used to form free monoidal categories~\cite{JS}. From this work came a veritable zoo of diagrammatic languages~\cite{selinger2009survey} for describing various flavours of monoidal categories. These languages, as a calculational tool for monoidal categories, have played a crucial role in the development of categorical quantum mechanics~\cite{AC2004,CoeckeDuncan2009} and the theory of quantum groups~\cite{TuraevBook}, as well as recently finding applications in computational linguistics~\cite{LambekVLambek,CoeckeLinguistics2010}.

While categorical string diagrams were developed very much in the same spirit as Penrose's notation, little work has been done formally relating abstract tensor systems to monoidal categories. This is the first contribution of this paper. In section~\ref{sec:assoc-cat}, we show that it is possible to construct a traced symmetric monoidal category from any abstract tensor system in a natural way. Furthermore, we show that reasoning with abstract tensors is sound and complete with respect to traced symmetric monoidal categories by showing that the associated category of the free abstract tensor system is, in fact, the free traced SMC.

It is generally well known that string diagrams are sound an complete for traced symmetric monoidal categories. This fact was alluded to in Joyal and Street's original paper, but the authors stopped short of providing a proof for the case where string diagrams had ``feedback'' (i.e. traced structure). Subsequently, this fact is often stated without proof~\cite{selinger2009survey}, sketched\footnote{The sketched proof of soundness/completeness for a diagrammatic language which the authors call \textit{sharing graphs} in~\cite{PlotkinComplete} is actually a special case of a theorem in Hasegawa's thesis~\cite{HasegawaThesis}, characterising the free cartesian-center traced symmetric monoidal category, which is proved in detail therein. Thanks to Masahito Hasegawa and Gordon Plotkin for pointing this out.}, or restricted to the case where the morphisms in the diagram have exactly one input and one output~\cite{KellyLaplaza1980,Shum1994}. Thus, the second contribution of this paper is a proof of soundness and completeness of the diagrammatic language as a corollary to the main theorem about freeness of the category generated by the free abstract tensor system.

% so in this sense stopped short of describing a category whose morphisms were string diagrams in their full generality.

% Concrete tensors also inherit addition from the base field. In the following section, we will show how this calculus can be generalised to express arbitrary processes with multiple inputs and outputs. In this case, there may not always be a natural interpretation for the sum of two tensor expressions (consider, e.g. classical digital circuits). In order to avoid ruling out such situations, we take this structure to be optional, and focus on the (more important) operations of product and contraction.

\section{Abstract tensor systems}

In this section, we shall define abstract tensor systems in a manner similar to Penrose's original definition in 1971~\cite{Penrose1971}. It will differ from the Penrose construction in two ways. First, we will not consider addition of tensors as a fundamental operation, but rather as extra structure that one could put on top (c.f. enriched category theory). Second, we will allow tensor inputs and outputs to have more than one type. In the previous section, we assumed for simplicity that the dimension of all indices was some fixed number $D$. We could also allow this to vary, as long as care is taken to only contract together indices of the same dimension. This yields a very simple example of a typed tensor system. Many other examples occur when we use tensors to study categories and operational physical theories.

For a set of types $\mathcal U = \{ A, B, C, \ldots \}$, fix a set of labels:
\[ \mathcal L = \{ a_1, a_2, \ldots, b_1, b_2, \ldots \} \]
and a typing function $\tau : \mathcal L \to \mathcal U$. We will always assume there are at least countably many labels corresponding to each type in $\mathcal U$. For finite subsets $\wx, \wy \in \labelsets$ and $\{ x_1, \ldots, x_N \} \subseteq \wx$ let:
\begin{equation}\label{eq:finite-fn}
  [x_1\mapsto y_1, x_2\mapsto y_2, \ldots, x_N \mapsto y_N] : \wx \to \wy
\end{equation}
be the function which sends each $x_i$ to $y_i$ and leaves all of the other elements of $\wx$ fixed.

% When two sets $\wx,\wy \in \labelsets$ are disjoint, we use juxtaposition to indicate their (disjoint) union:
% \begin{align*}
%   \wx\wy & = \wy\wx = \wx \cup \wy    & \textrm{ \ \ (if $\wx \cap \wy = \emptyset$)} \\
%   a \wx  & = \wx a = \wx \cup \{ a \} & \textrm{ \ \ (if $a \notin \wx$)}
% \end{align*}

% When there can be no confusion, names and name sets occurring in juxtapositions are assumed to be disjoint without further comment.

\begin{definition}\label{def:abstract-tensor-system}
  A (small) \textit{abstract tensor system} consists of:

  \begin{itemize}
    \item A set $\mathcal T(\wx, \wy)$ for all $\wx, \wy \in \labelsets$ such that $\wx \cap \wy = \emptyset$,
    \item a \textit{tensor product} operation:
      \[ (- \cdot -) : \mathcal T(\wx,\wy) \times \mathcal T(\wx', \wy') \to \mathcal T(\wx \cup \wx', \wy \cup \wy') \]
      defined whenever $(\wx \cup \wy) \cap (\wx' \cup \wy') = \emptyset$
    \item a \textit{contraction} function for $a \in \wx, b\in \wy$ such that $\tau(a) = \tau(b)$:
      \[ \K_a^b : \mathcal T(\wx, \wy) \to \mathcal T(\wx - \{ a \}, \wy - \{ b \}) \]
    \item a bijection between sets of tensors, called a \textit{relabelling}:
    \[ r_f : \mathcal T(\wx,\wy) \overset{\sim}{\to} \mathcal T(f(\wx), f(\wy)) \]
    for every bijection $f : (\wx \cup \wy) \overset{\sim}{\to} \wz$ such that $\tau(f(x)) = \tau(x)$ for all $x \in \wx \cup \wy$,
    \item a chosen tensor called a \textit{$\delta$-element} $\delta_a^b \in \mathcal T(\{a\},\{b\})$ for all $a,b$ such that $\tau(a) = \tau(b)$, along with an ``empty'' $\delta$-element $1 \in \mathcal T(\emptyset,\emptyset)$
  \end{itemize}
\end{definition}

% \footnote{The commutativity of $\otimes$ may seem odd at first, given the usual tensor product is non-commutative. However, as noted by Penrose, this corresponds to the concrete equation $\psi^i \phi^j = \phi^j \psi^i$ ($\neq \psi^j \phi^i$). That is, general tensors on non-symmetric on indices, but the order in which products are written is irrelevant.}

\noindent Before giving the axioms, we introduce some notation. Let $\psi\phi := \psi \cdot \phi$ and $\psi[f] := r_f(\psi)$. For $\psi \in \mathcal T(\wx,\wy)$ let $L(\psi) = \wx \cup \wy$. If a label $a$ is in $L(\psi)$, we say $a$ is \textit{free} in $\psi$. If a label occurs in a contraction, we say it is \textit{bound}. Using this notation, the axioms of an abstract tensor system are as follows:

% $L_*(\psi) = \wx$, $L^*(\psi) = \wy$ and

  \begin{enumerate}
    \item[T1.] $\K_a^b(\K_{a'}^{b'}(\psi)) = \K_{a'}^{b'}(\K_a^b(\psi))$
    \item[T2.] $(\psi\phi)\xi = \psi(\phi\xi)$, $\psi 1 = \psi = 1 \psi$, and $\psi\phi = \phi\psi$
    \item[T3.] $\K_a^b(\psi\phi) = (\K_a^b(\psi))\phi$ for $a,b \notin L(\phi)$
    \item[T4.] $\K_a^b(\delta_a^c \psi) = \psi[b \mapsto c]$ and $\K_b^c(\delta_a^c \psi) = \psi[b \mapsto a]$
  % \end{enumerate}
  % \noindent and three relabelling axioms:
  % \begin{enumerate}
    \item[L1.] $\psi[f][g] = \psi[g \circ f]$ and $\psi[\textrm{id}] = \psi$
    \item[L2.] $(\psi[f])\phi = (\psi\phi)[f]$ where $\textit{cod}(f) \cap L(\phi) = \emptyset$
    \item[L3.] $\K_a^b(\psi)[f'] = \K_{f(a)}^{f(b)}(\psi[f])$ where $f'$ is the restriction of $f$ to $L(\psi) - \{ a, b \}$
    \item[L4.] $\delta_a^b[a \mapsto a', b \mapsto b'] = \delta_{a'}^{b'}$
  \end{enumerate}

Note that L3 implies in particular that the choice of bound labels in irrelevant to the value of a contracted tensor.

\begin{lemma}\rm \label{lem:bound-labels}
  Let $\psi$ be a tensor containing a lower label $a$ and upper label $b$, and let $a', b'$ be distinct labels not occurring in $L(\psi)$ such that $\tau(a) = \tau(a')$ and $\tau(b) = \tau(b')$. Then
  \[ \K_a^b(\psi) = \K_{a'}^{b'}(\psi[a \mapsto a', b \mapsto b']) \]
\end{lemma}

\begin{proof}
  Let $f = [a \mapsto a', b \mapsto b']$ and note that the restriction of $f$ to $L(\psi) - \{ a, b \}$ is the identity map. Then:
  \[ \K_{a'}^{b'}(\psi[a \mapsto a', b \mapsto b']) = \K_{f(a)}^{f(b)}(\psi[f]) =
     \K_{a}^{b}(\psi)[\textrm{id}] = \K_{a}^{b}(\psi) \qedhere \]
\end{proof}

% \begin{remark}
%   Abstract tensor systems are defined relative to a \textit{set} of types, so they should really be called \textit{small} abstract tensor systems. For a proper class of types, it is still possible to make a reasonable definition by replacing a the global set of labels with local labellings $\mathcal L$ for sets $\mathcal T_{\mathcal L}(\wx, \wy)$ of tensors labelled by $\mathcal L$, with a suitable generalisation of the notion of relabelling. As the abstract tensor systems we are primarily concerned with are all small, we will avoid this extra complexity.
% \end{remark}

\subsection{Einstein notation and free abstract tensor systems}

$\mathcal T(\wx, \wy)$ is just an abstract set. Its elements should be thought of ``black boxes'' whose inputs are labelled by the set $\wx$ and whose outputs are labelled by the set $\wy$. Despite this sparse setting, working with abstract tensors is no more difficult than working with concrete tensors, with the help of some suggestive notation.

First, let a \textit{tensor symbol} $\Psi = (\psi, \vx, \vy)$ be a triple consisting of a tensor $\psi \in \mathcal T(\wx, \wy)$ and lists $\vx, \vy$ with no repetition such that the elements of $\vx$ are precisely $\wx$ and the elements of $\vy$ are precisely $\wy$. Equivalently, a tensor symbol is a tensor along with a total ordering on input and output labels.

\begin{notation}
  Let $\vx = [x_1, \ldots, x_m]$ and $\vy = [y_1, \ldots, y_n]$ be lists of labels. Then we write the tensor symbol $\Psi = (\psi, \vx, \vy)$ as:
  \[ \psi_{\vx}^{\vy} \quad \textrm{ or } \quad \psi_{x_1, \ldots, x_m}^{y_1, \ldots, y_n} \]
  If $m = n$ and $\tau(x_i) = \tau(y_i)$, then let:
  \[ \delta_{\vx}^{\vy} := \delta_{x_1}^{y_1} \ldots \delta_{x_n}^{y_n} \]
  In particular, the above expression evaluates to $1 \in \mathcal T(\emptyset, \emptyset)$ when $\vx = \vy = []$.
\end{notation}

An \textit{alphabet} $\mathcal A$ for an abstract tensor system is a set of tensor symbols such that for all $\wx,\wy$ each element $\psi \in \mathcal T(\wx,\wy)$ occurs at most once.

The fact that labels in a tensor symbol are ordered may seem redundant, given that the labels themselves identify inputs and outputs in $\psi$. However, it gives us a convenient (and familiar) way to express relabellings. Given $\psi_{x_1,\ldots,x_m}^{y_1,\ldots,y_n} \in \mathcal A$, we can express a relabelled tensor as:
\begin{equation}\label{eq:eins-relabel}
  \llbracket \psi_{a_1,\ldots,a_m}^{b_1,\ldots,b_n} \rrbracket =
  \psi[x_1 \mapsto a_1, \ldots, x_m \mapsto a_m,
       y_1 \mapsto b_1, \ldots, y_n \mapsto b_n]
\end{equation}

It will often be convenient to refer to arbitrary tensors $\psi_{x_1,\ldots,x_m}^{y_1,\ldots,y_n} \in \mathcal T(\wx,\wy)$ using tensor symbols. In this case, we treat subsequent references to $\psi$ (possibly with different labels) as tensors that have been relabelled according to (\ref{eq:eins-relabel}).

\begin{definition}
  An \textit{Einstein expression} over an alphabet $\mathcal A$ is a list of $\delta$-elements and (possibly relabelled) tensor symbols, where each label is either distinct or occurs as a repeated upper and lower label.
\end{definition}

For an Einstein expression $E$, let $E_{[a \mapsto a']}$ and $E^{[a \mapsto a']}$ be the same expression, but with a lower or upper label replaced. Einstein expressions are interpreted as abstract tensors as follows. First, any repeated labels are replaced by new, fresh labels of the same type, along with contractions.
\[ \llbracket E \rrbracket = \K_{a}^{\overline{a}}(
   \llbracket E^{[a \mapsto \overline{a}]} \rrbracket) \qquad
   \textrm{where $a$ is repeated and $\overline{a}$ is not in $E$} \]
Once all labels are distinct, juxtaposition is treated as tensor product:
\[
  \llbracket E E' \rrbracket = \llbracket E \rrbracket \llbracket E' \rrbracket \qquad
  \textrm{where $EE'$ has no repeated labels} 
\]
Single tensor symbols are evaluated as in equation (\ref{eq:eins-relabel}), and empty expressions evaluate to $1 \in \mathcal T(\emptyset,\emptyset)$. We will often suppress the semantic brackets $\llbracket - \rrbracket$ when it is clear we are talking about equality of tensors rather than syntactic equality of Einstein expressions.

\begin{theorem}\rm
  An Einstein expression unambiguously represents an element of some set $\mathcal T(\wx,\wy)$. Furthermore, any expression involving constants taken from the alphabet $\mathcal A$, labellings, tensor products, and contractions can be expressed this way.
\end{theorem}

\begin{proof}
  First, we show that the expression $E$ represents an abstract tensor without ambiguity. In the above prescription, the choices we are free to make are (1) the order in which contractions are performed, (2) the choice of fresh labels $\overline a$, and (3) the order in which tensor products are evaluated. However, (1) is irrelevant by axiom T1 of an abstract tensor system, (2) by Lemma~\ref{lem:bound-labels}, and (3) by axiom T2.
  
  For the other direction, suppose $e$ is an expression involving constants, relabellings, tensor product, and contraction. Then, we can use the axioms of an abstract tensor system to pull contractions to the outside and push labellings down to the level of constants. Then, by the axioms of an abstract tensor system, there is an equivalence of expressions:
  \[ e \equiv \K_{a_1}^{b_1}(\K_{a_1}^{b_1}(\ldots (\K_{a_n}^{b_n}(
       \psi_1[f_1] \psi_2[f_2] \ldots \psi_m[f_m]
     )))\ldots) \]
  Let $\Psi_i$ be the tensor symbol corresponding to the relabelled constant $\psi_i[f_i]$. Then, there is an equality of tensors: 
  \( e = \llbracket (\Psi_1 \Psi_2 \ldots \Psi_m)^{[b_1 \mapsto a_1, \ldots, b_n \mapsto a_n]} \rrbracket \)
\end{proof}

By Lemma~\ref{lem:bound-labels}, the particular choice of bound labels in an Einstein expression is irrelevant. That is, $\llbracket E \rrbracket = \llbracket E_{[x \mapsto \overline x]}^{[x \mapsto \overline x]} \rrbracket$, for $x$ a bound label and $\overline x$ a new fresh label such that $\tau(x) = \tau(\overline x)$. Also note that, by axiom T4 of an abstract tensor system, it is sometimes possible to eliminate $\delta$ elements from Einstein expressions.
\[ \llbracket E \delta_a^b \rrbracket = \llbracket E^{[a \mapsto b]} \rrbracket \qquad
   \textrm{if $E$ contains $a$ as an upper label} \]
and similarly:
\[ \llbracket E \delta_a^b \rrbracket = \llbracket E_{[b \mapsto a]} \rrbracket \qquad
   \textrm{if $E$ contains $b$ as a lower label} \]
   
The only cases where such a reduction is impossible are (1) when neither label on $\delta$ is repeated, or (2) when the repeated label is on $\delta$ itself: $\delta_a^a$. For reasons that will become clear in the graphical notation, case (1) is called a \textit{bare wire} and case (2) is called a \textit{circle}. If no $\delta$-elimination is possible, an expression $E$ is called a \textit{reduced Einstein expression}. Up to permutation of tensor symbols and renaming of bound labels, this reduced form is unique.

We are now ready to define the \textit{free abstract tensor system} over an alphabet $\mathcal A$. First, define an equivalence relation on Einstein expressions. Let $E \approx E'$ if $E$ can be obtained from $E'$ by permuting tensor symbols, adding or removing $\delta$-elements as above, or renaming repeated labels. Let $|E|$ be the $\approx$-equivalence class of $E$. Then, the free abstract tensor system $\atsFree$ is defined as follows.
\begin{itemize}
  \item $\mathcal T(\wx, \wy)$ is the set of $|E|$ where $\wx$ and $\wy$ occur as non-repeated lower and upper labels in $E$, respectively
  \item $|E| \cdot |E'| = |EE'|$, with $E$ and $E'$ chosen with no common labels
  \item $\K_a^b(|E|) = |E^{[b \mapsto a]}|$
  \item $|E|[f] = |E_{[f']}^{[f']}|$ where $f'$ sends bound labels to new fresh labels (i.e. not in $\cod(f)$) of the correct type and acts as $f$ otherwise
\end{itemize}

Since we assume infinitely many labels, it is always possible to find fresh labels. Furthermore, the three tensor operations do not depend on the choice of (suitable) representative. Note, it is also possible to define the free ATS in terms of reduced Einstein expressions, in closer analogy with free groups. However, it will be convenient in the proof of Theorem~\ref{thm:free-tsmc} to let $|E|$ contain non-reduced expressions as well.

\section{Diagrammatic notation}\label{sec:diagrams}

There is another, often more convenient, alternative to Einstein notation for writing tensor expressions: \textit{string diagram} notation. Tensor symbols are represented as boxes with one input for each lower index and one output for each upper index.
\[
\psi_{a,b}^c \ \ \Rightarrow\ \ 
\beginpgfgraphicnamed{single_box}
\begin{tikzpicture}[dotpic]
	\begin{pgfonlayer}{nodelayer}
		\node [style=square box, minimum width=1 cm] (0) at (0, 0) {$\psi$};
		\node [style=none] (1) at (-0.5, -0.25) {};
		\node [style=none] (2) at (0.5, -0.25) {};
		\node [style=none] (3) at (1, -1.25) {};
		\node [style=none] (4) at (-1, -1.25) {};
		\node [style=none] (5) at (0, 0.25) {};
		\node [style=none] (6) at (0, 1.25) {};
		\node [style=none, font=\footnotesize, anchor=west, xshift=-1 mm] (7) at (-0.75, -1.25) {$a\!:\!A$};
		\node [style=none, font=\footnotesize, anchor=west, xshift=-1 mm] (8) at (1.25, -1.25) {$b\!:\!B$};
		\node [style=none, font=\footnotesize, anchor=west, xshift=-1 mm] (9) at (0.25, 1.25) {$c\!:\!C$};
	\end{pgfonlayer}
	\begin{pgfonlayer}{edgelayer}
		\draw [in=-90, out=90] (4.center) to (1.center);
		\draw [in=-90, out=90] (3.center) to (2.center);
		\draw (5.center) to (6.center);
	\end{pgfonlayer}
\end{tikzpicture}}
\endpgfgraphicnamed
\]
These inputs and outputs are marked with both a label and the type of the label, but this data is often suppressed if it is irrelevant or clear from context. A repeated label is indicated by connecting an input of one box to an output of another. These connections are called \textit{wires}.
\begin{equation}
\psi_{a,b}^c \phi_d^{b,e} \ \ \Rightarrow\ \ 
\beginpgfgraphicnamed{two_boxes}
\begin{tikzpicture}[dotpic]
	\begin{pgfonlayer}{nodelayer}
		\node [style=square box, minimum width=1 cm] (0) at (-0.75, 1) {$\psi$};
		\node [style=none] (1) at (-0.25, 0.75) {};
		\node [style=none] (2) at (0.5, -0.75) {};
		\node [style=none] (3) at (-0.75, 1.25) {};
		\node [style=none] (4) at (-0.75, 2) {};
		\node [style=square box, minimum width=1 cm] (5) at (1, -1.25) {$\phi$};
		\node [style=none] (6) at (1.75, 0.5) {};
		\node [style=none] (7) at (1.5, -0.75) {};
		\node [style=none] (8) at (1, -1.75) {};
		\node [style=none] (9) at (0.75, -2.5) {};
		\node [style=none] (10) at (-1.25, 0.75) {};
		\node [style=none] (11) at (-1.75, -0.5) {};
		\node [style=none, font=\footnotesize, anchor=west, xshift=-1 mm] (12) at (-0.5, 2) {$c\!:\!C$};
		\node [style=none, font=\footnotesize, anchor=west, xshift=-1 mm] (13) at (-1.5, -0.5) {$a\!:\!A$};
		\node [style=none, font=\footnotesize, anchor=west, xshift=-1 mm] (14) at (0.5, 0) {$B$};
		\node [style=none, font=\footnotesize, anchor=west, xshift=-1 mm] (15) at (1, -2.5) {$e\!:\!E$};
		\node [style=none, font=\footnotesize, anchor=west, xshift=-1 mm] (16) at (2, 0.5) {$d\!:\!D$};
	\end{pgfonlayer}
	\begin{pgfonlayer}{edgelayer}
		\draw [in=-82, out=90, looseness=0.75] (2.center) to (1.center);
		\draw (3.center) to (4.center);
		\draw [bend right=15] (6.center) to (7.center);
		\draw [in=90, out=-90] (8.center) to (9.center);
		\draw [in=-90, out=90] (11.center) to (10.center);
	\end{pgfonlayer}
\end{tikzpicture}}
\endpgfgraphicnamed
\end{equation}

Repeated labels are not written, as they are irrelevant by Lemma~\ref{lem:bound-labels}. $\delta$-elements are written as \textit{bare wires}, i.e. wires that are not connected to any boxes. In particular, contracting the input and the output of a $\delta$ element together yields a circle. Also, connecting a box to a bare wire has no effect, which is consistent with axiom T4.
\[
\delta_a^b\ =\ %
\beginpgfgraphicnamed{bare_wire}
\begin{tikzpicture}[dotpic]
	\begin{pgfonlayer}{nodelayer}
		\node [style=none, font=\footnotesize, yshift=1 pt, xshift=1 pt] (0) at (0.5, 0.75) {$b$};
		\node [style=none] (1) at (0.25, 0.75) {};
		\node [style=none, font=\footnotesize] (2) at (0, -0.75) {$a$};
		\node [style=none] (3) at (-0.25, -0.75) {};
	\end{pgfonlayer}
	\begin{pgfonlayer}{edgelayer}
		\draw [in=-90, out=90] (3.center) to (1.center);
	\end{pgfonlayer}
\end{tikzpicture}}
\endpgfgraphicnamed \qquad
\delta_a^a\ =\ %
\beginpgfgraphicnamed{circle}
\begin{tikzpicture}[dotpic]
	\begin{pgfonlayer}{nodelayer}
		\node [style=none] (0) at (-0.75, 0) {};
		\node [style=none] (1) at (0, 0.75) {};
		\node [style=none] (2) at (0, -0.75) {};
		\node [style=none] (3) at (0.75, 0) {};
	\end{pgfonlayer}
	\begin{pgfonlayer}{edgelayer}
		\draw [in=0, out=90] (3.center) to (1.center);
		\draw [in=90, out=180] (1.center) to (0.center);
		\draw [in=-90, out=0] (2.center) to (3.center);
		\draw [in=180, out=-90] (0.center) to (2.center);
	\end{pgfonlayer}
\end{tikzpicture}}
\endpgfgraphicnamed \qquad
\psi_{a,b}^d \delta_d^c\ =\ \ %
\beginpgfgraphicnamed{single_box_id}
\begin{tikzpicture}[dotpic]
	\begin{pgfonlayer}{nodelayer}
		\node [style=square box, minimum width=1 cm] (0) at (0, -0.25) {$\psi$};
		\node [style=none] (1) at (-0.5, -0.5) {};
		\node [style=none] (2) at (0.5, -0.5) {};
		\node [style=none] (3) at (0.75, -1.5) {};
		\node [style=none] (4) at (-0.75, -1.5) {};
		\node [style=none] (5) at (0, 0) {};
		\node [style=none] (6) at (0, 1.5) {};
		\node [style=none, font=\footnotesize] (7) at (0.25, 1.5) {$c$};
		\node [style=none, font=\footnotesize] (8) at (-0.5, -1.5) {$a$};
		\node [style=none, font=\footnotesize] (9) at (1, -1.5) {$b$};
		\node [style=none] (10) at (-0.25, 0.75) {};
		\node [style=none] (11) at (0.25, 0.75) {};
	\end{pgfonlayer}
	\begin{pgfonlayer}{edgelayer}
		\draw [in=-90, out=90] (4.center) to (1.center);
		\draw [in=-90, out=90] (3.center) to (2.center);
		\draw (5.center) to (6.center);
		\draw [style=dashed edge] (10.center) to (11.center);
	\end{pgfonlayer}
\end{tikzpicture}}
\endpgfgraphicnamed\ =\ \ %
\beginpgfgraphicnamed{single_box_simple}
\begin{tikzpicture}[dotpic]
	\begin{pgfonlayer}{nodelayer}
		\node [style=square box, minimum width=1 cm] (0) at (0, 0) {$\psi$};
		\node [style=none] (1) at (-0.5, -0.25) {};
		\node [style=none] (2) at (0.5, -0.25) {};
		\node [style=none] (3) at (0.75, -1.25) {};
		\node [style=none] (4) at (-0.75, -1.25) {};
		\node [style=none] (5) at (0, 0.25) {};
		\node [style=none] (6) at (0, 1) {};
		\node [style=none, font=\footnotesize] (7) at (0.25, 1) {$c$};
		\node [style=none, font=\footnotesize] (8) at (-0.5, -1.25) {$a$};
		\node [style=none, font=\footnotesize] (9) at (1, -1.25) {$b$};
	\end{pgfonlayer}
	\begin{pgfonlayer}{edgelayer}
		\draw [in=-90, out=90] (4.center) to (1.center);
		\draw [in=-90, out=90] (3.center) to (2.center);
		\draw (5.center) to (6.center);
	\end{pgfonlayer}
\end{tikzpicture}}
\endpgfgraphicnamed\ =\ \psi_{a,b}^c
\]

The most important feature of the diagrammatic notation is that only the connectivity of the diagram matters. Therefore, the value is invariant under topological deformations. For example:
\ctikzfig{deformation}

\begin{theorem} \rm
  For an alphabet $\mathcal A$, any tensor in $\atsFree$ can be unambiguously represented in the diagrammatic notation.
\end{theorem}

\begin{proof}
  For a diagram $D$, form $E$ as follows. First, chose a label that does not already occur in $D$ for every wire connecting two boxes and for every circle. Then, let $E = \Psi_1 \ldots \Psi_n$, where each $\Psi_i$ is a box from $D$, with labels taken from the input and output wires. Then, $D$ represents the $\approx$-equivalence class of $E$ defined in the previous section. By definition of $|E|$, the choice of labels for previously unlabelled wires in $D$ and the order of the $\Psi_i$ are irrelevant. Thus, $D$ defines precisely one equivalence class $|E|$ in this manner.
\end{proof}

\section{Traced symmetric monoidal categories}

A monoidal category $(\mathcal C, \otimes, I, \alpha, \lambda, \rho)$ is a category that has a horizontal composition operation $\otimes : \mathcal C \times \mathcal C \to \mathcal C$ called the \textit{monoidal product} that is associative and unital (up to isomorphism) and interacts well with the categorical (aka vertical) composition. A \textit{strict} monoidal category is a monoidal category such that the natural isomorphisms $\alpha, \lambda, \rho$ are all identity maps. A \textit{symmetric} monoidal category has an additional swap map $\sigma_{A,B} : A \otimes B \to B \otimes A$ such that $\sigma_{A,B} = \sigma_{B,A}$, and it interacts well with the rest of the monoidal structure. For full details, see e.g.~\cite{MacLane}.

\begin{definition}\label{def:traced-category}
  A traced symmetric monoidal category $\mathcal C$ is a symmetric monoidal category with a function
  \[ \Tr^X : \hom_{\mathcal C}(A \otimes X, B \otimes X) \rightarrow \hom_{\mathcal C}(A, B) \]
  defined for all objects $A, B, X$, satisfying the following five axioms:\footnote{Note that some structure isomorphisms have been suppressed for clarity. The coherence theorem for monoidal categories lets us do this without ambiguity.}
  \begin{enumerate}
    \item $\Tr^X((g \otimes X) \circ f \circ (h \otimes X)) = g \circ \Tr^X(f) \circ h$
    \item $\Tr^Y(f \circ (A \otimes g)) = \Tr^X((B \otimes g) \circ f)$
    \item $\Tr^I(f) = f$ and $\Tr^{X \otimes Y}(f) = \Tr^X(\Tr^Y(f))$
    \item $\Tr^X(g \otimes f) = g \otimes \Tr^X(f)$
    \item $\Tr^X(\sigma_{X,X}) = 1_X$
  \end{enumerate}
\end{definition}

Just as monoidal categories have strict and non-strict versions, so too do monoidal functors. \textit{Strict} (traced, symmetric) monoidal functors preserve all of the categorical structure up to equality, whereas \textit{strong} functors preserve all of the structure up to coherent natural isomorphism. The term ``strong'' is used by way of contrast with \textit{lax} monoidal functors, which preserve the structure only up to (possibly non-invertible) natural transformations. Again, see~\cite{MacLane} for full definitions.

Let \catTSMC be the category of traced symmetric monoidal categories and strong monoidal functors that preserve symmetry maps and the trace operation, and let $\catTSMC_s$ be the strict version.

\subsection{The free traced symmetric monoidal category}

Two morphisms are equal in a free (symmetric, traced, compact closed, etc.) monoidal category \textit{if and only if} their equality can be established only using the axioms of that category. Thus free monoidal categories are a powerful tool for proving theorems which hold in \textit{all} categories of a particular kind. Free monoidal categories are defined over a collection of generators called a \textit{monoidal signature}.

\begin{notation}
  For a set $X$, let $X^*$ be the free monoid over $X$, i.e. the set of lists with elements taken from $X$ where multiplication is concatenation and the unit is the empty list. For a function $f : X \rightarrow Y$, let $f^* : X^* \rightarrow Y^*$ be the lifting of $f$ to lists:
\( f^*([x_1, \ldots, x_n]) = [f(x_1), \ldots, f(x_n)] \).
\end{notation}

\begin{definition}\label{def:monoidal-signature}
  A (small, strict) monoidal signature $T = (O, M, \dom, \cod)$ consists of a set of objects $O$, a set of morphisms $M$, and a pair of functions $\dom : M \rightarrow O^*$ and $\cod : M \rightarrow O^*$.
\end{definition}

The maps $\dom$ and $\cod$ should be interpreted as giving input and output types to a morphism $m \in M$. For instance, if $\dom(m) = [A,B,C]$ and $\cod(m) = [D]$, then $m$ represents a morphism $m : A \otimes B \otimes C \to D$. The empty list is interpreted as the tensor unit $I$.

% \begin{example}
%   Define a monoidal signature $T = (O, M, \dom, \cod)$ where
%   \begin{align*}
%     O    & = \{ A, B, C \} \\
%     M    & = \{ f, g \} \\
%     \dom & = \{ f \mapsto [A,B],\ g \mapsto [C] \} \\
%     \cod & = \{ f \mapsto [C],\ g \mapsto [C] \}
%   \end{align*}
  
%   This signature defines three (primitive) objects $A, B, C$ and two morphisms $f : A \otimes B \rightarrow C$ and $g : C \rightarrow C$.
% \end{example}

There is also a notion of a non-strict monoidal signature. In that case, $O^*$ is replaced with the free $(\otimes, I)$-algebra over $O$. However, by the coherence theorem of monoidal categories, there is little difference between strict monoidal signatures and non-strict monoidal signatures with some fixed choice of bracketing.

\begin{definition}\label{def:monsig}
  For monoidal signatures $S$, $T$, a monoidal signature homomorphism $f$ consists of functions $f_O : O_{S} \rightarrow O_{T}$ and $f_M : M_{S} \rightarrow M_{T}$ such that $\dom_T \circ f_M = f_O^* \circ \dom_S$ and $\cod_T \circ f_M = f_O^* \circ \cod_S$. $\catMonSig$ is the category of monoidal signatures and monoidal signature homomorphisms.
\end{definition}

  %the following diagrams commute.
  % \begin{center}
  %   \begin{tikzpicture}
  %     \matrix (m) [cdiag] {
  %       M_{S} & O_{S}^* \\
  %       M_{T} & O_{T}^* \\
  %     };
  %     \path [arrs]
  %       (m-1-1) edge node {$\dom_{S}$} (m-1-2)
  %       (m-2-1) edge node [swap] {$\dom_{T}$} (m-2-2)
  %       (m-1-1) edge node [swap] {$f_M$} (m-2-1)
  %       (m-1-2) edge node {$f_O^*$} (m-2-2);
  %   \end{tikzpicture}
  %   \qquad\qquad
  %   \begin{tikzpicture}
  %     \matrix (m) [cdiag] {
  %       M_{\mathcal S} & O_{S}^* \\
  %       M_{\mathcal T} & O_{T}^* \\
  %     };
  %     \path [arrs]
  %       (m-1-1) edge node {$\cod_{S}$} (m-1-2)
  %       (m-2-1) edge node [swap] {$\cod_{T}$} (m-2-2)
  %       (m-1-1) edge node [swap] {$f_M$} (m-2-1)
  %       (m-1-2) edge node {$f_O^*$} (m-2-2);
  %   \end{tikzpicture}
  % \end{center}

A monoidal signature is essentially a strict monoidal category without composition or identity maps. A monoidal signature homomorphism is thus a strict monoidal functor, minus the condition that it respect composition and identity maps.

There is an evident forgetful functor from $\catTSMC_s$ into $\catMonSig$, by throwing away composition. If this forgetful functor has a left adjoint $F$, the image of a signature $T$ under $F$ is called the \textit{free strict monoidal category} over $T$.

However, when considering the free non-strict category, the issue becomes a bit delicate. In particular, it is no longer reasonable to expect the lifted morphism $\widetilde v$ to be unique \textit{on the nose}, but rather unique up to coherent natural isomorphism. Thus, the adjunction $\catMonSig \dashv \catTSMC_s$ should be replaced with a pseudo-adjunction of some sort. To side-step such higher categorical issues, Joyal and Street simply state the appropriate correspondence between valuations of a signature and strong symmetric monoidal functors from the free category~\cite{JS}. Here, we state the traced version of their definition. Let $[T, \mathcal C]$ be the category of valuations of $T$ in $\mathcal C$ and $\catTSMC(\mathcal C, \mathcal D)$ be the category of strong traced symmetric monoidal functors from $\mathcal C$ to $\mathcal D$ and monoidal natural isomorphisms.

\begin{definition}
  For a monoidal signature \catSig, a traced symmetric monoidal category $\mathbb F(T)$ is called the free traced SMC when, for any traced SMC $\mathcal C$, there exists a valuation $\eta \in \textrm{ob}([\catSig, \mathcal C])$ such that:
  \[ ( - \circ \eta ) : \catTSMC(\mathbb F(T), \mathcal C) \to [T, \mathcal C] \]
  yields an equivalence of categories.
\end{definition}

This equivalence of categories plays an analogous role to the isomorphism of hom-sets characterising an adjunction. For brevity, we omit the definitions of $[T, \mathcal C]$ and $( - \circ \eta )$. The first represents the category of valuations of a monoidal signature $T$ into a (possibly non-strict) monoidal category $\mathcal C$ and valuation morphisms (i.e. the valuation analogue of natural transformations). The latter represents the natural way to ``compose'' a valuation $\eta$ with a strong monoidal functor to yield a new valuation. Details can be found in~\cite{JS}.

% \begin{definition}
%   A monoidal signature is called \textit{simple} if the images of $\dom$ and $\cod$ are restricted to single-element lists.
% \end{definition}

% In \cite{KellyLaplaza1980}, Kelly and Laplaza gave a prescription for constructing the free category on any ``algebraically-defined'' additional structure on a category. They went on to describe concretely the free compact closed category on a category (or equivalently, a simple signature). Shum proved a similar result in 1994~\cite{Shum1994} for tortile categories, i.e. braided monoidal categories with coherently-defined left and right duals. In the following sections, we will show how to define the free traced symmetric monoidal category on non-simple signatures.

\section{The traced SMC of an abstract tensor system}\label{sec:assoc-cat}

% \begin{itemize}
%   \TODOitem{add informal version before formal construction}
% \end{itemize}

In this section, we construct the associated traced symmetric monoidal category of an abstract tensor system. We shall see that an abstract tensor system and a traced SMC are essentially two pictures of the same thing. However, these two pictures vary in how they refer to the inputs and outputs of maps. On the one hand, traced SMCs take the input and output ordering to be fixed, and rely on structural isomorphisms for shuffling inputs and outputs around. On the other, abstract tensor systems refer to inputs and outputs by \textit{labels}, but care must be taken to make sure these labels are given in a consistent manner.

Let $\mathbb N$ be the natural numbers and $\mathbb B = \{ 0, 1 \}$. From hence forth, we will assume that the set $\mathcal L$ of labels has a special subset $\mathcal L_c \cong \mathcal U \times \mathbb N \times \mathbb B$ called the \textit{canonical} labels. We write the elements $(X,i,0)$ and $(X,i,1)$ as $x_i^\ii$ and $x_i^\oo$, respectively. Then, let:
\[ \tau(x_i^\ii) = \tau(x_i^\oo) = X \]

As we shall see in definition~\ref{def:assoc-cat}, using canonical labels allows us to impose input and output ordering on a tensor in order to treat it as a morphism in a monoidal category. It also yields a natural choice of free labels in the monoidal product of two tensors, which is an important consideration when the labels of the two tensors being combined are not disjoint.

\begin{notation}
  For $\vec X = [X_1,X_2,\ldots,X_N]$ a list of types, define the following list of labels for $1 \leq m < n \leq N$ and $i = 0,1$:
  \begin{align*}
    \vx_{m..n}^{(i)} & := [x_m^{(i)}, x_{m+1}^{(i)}, \ldots, x_{n-1}^{(i)}, x_n^{(i)}]
  \end{align*}
  The set containing the above elements is denoted $\wx_{m..n}^{(i)}$. In the case where $m = 1$ and $n = \textrm{length}(\vec X)$, we often omit the subscripts, writing simply $\vx^{(i)}$ and $\wx^{(i)}$.
\end{notation}

\begin{definition} \label{def:assoc-cat}
Let $\mathcal S = (\mathcal U, \mathcal L, \mathcal T(-,-))$ an abstract tensor system with a choice of canonical labels $\mathcal L_c \subseteq \mathcal L$. Then \CS is the traced symmetric monoidal category defined as follows:
\begin{align*}
  \textrm{ob}(\CS)        & = \mathcal U^* \\
  \hom_{\CS}(\vec X, \vec Y) & =
    \mathcal T(\wx^\ii, \wy^\oo) \\
  \vec X \otimes \vec Y & = \vec X \vec Y \quad (I = [\,])
\end{align*}
For $\psi : \vec X \to \vec Y$, $\phi : \vec Y \to \vec Z$, $\widetilde \psi : \vec U \to \vec V$, and $\xi : \vec U \otimes \vec X \to \vec V \otimes \vec X$, the rest of the structure is defined as:
\begin{align*}
  \phi_{\vy^\ii}^{\vz^\oo} \circ \psi_{\vx^\ii}^{\vy^\oo} & =
    \psi_{\vx^\ii}^{\vec{y'}} \phi_{\vec{y'}}^{\vz^\oo} \\
  \psi_{\vx^\ii}^{\vy^\oo} \otimes \widetilde \psi_{\,\vu^\ii}^{\,\vv^\oo} & =
    \psi_{\vx_{1..m}^\ii}^{\vy_{1..n}^\oo}
    \widetilde \psi_{\,\vu_{m+1..m+m'}^\ii}^{\,\vv_{n+1..n+n'}^\oo} \\
  \textrm{id}_{\vec X} & = \delta_{\vx^\ii}^{\vx^\oo} \\
  \sigma_{\vec X, \vec Y} & = 
    \delta_{\vx_{1..m}^\ii}^{\vx_{n+1..n+m}^\oo}
    \delta_{\vy_{m+1..m+n}^\ii}^{\vy_{1..n}^\oo} \\
  \textrm{Tr}^{\vec X}(\xi_{\vu_{1..m}^\ii \vx_{m+1..m+k}^\ii}^{\vv_{1..n}^\oo \vx_{n+1..n+k}^\oo}) & =
    \xi_{\vu_{1..m}^\ii\vec{x'}}^{\vv_{1..n}^{(1)}\vec{x'}}
\end{align*}
where $\vec{x'}$ and $\vec{y'}$ are chosen as fresh (possibly non-canonical) labels.
\end{definition}

\begin{theorem} \rm
  \CS is a strict, traced symmetric monoidal category.
\end{theorem}

\begin{proof}
Associativity follows from ATS axioms (used implicitly in the Einstein notation):
\[
  (\xi \circ \phi) \circ \psi =
   \psi_{\vx^{(0)}}^{\vec{y'}} \phi_{\vec{y'}}^{\vec{z'}} \xi_{\vec{z'}}^{\vw^{(1)}} =
  \xi \circ (\phi \circ \psi)
\]
and similarly for identity maps. Associativity and unit laws of the monoidal product follow straightforwardly from associativity of $(- \cdot -)$. The interchange law can be shown as:
\begin{align*}
  (\widetilde \psi_{\vy^{(0)}}^{\vz^{(1)}} \otimes \widetilde \phi_{\vv^{(0)}}^{\vw^{(1)}}) \circ
  (\psi_{\vx^{(0)}}^{\vy^{(1)}} \otimes \phi_{\vu^{(0)}}^{\vv^{(1)}})
  & =
  \psi_{\vx_{1..m}^{(0)}}^{\vec{w'}} \phi_{\vu_{m+1..m+m'}^{(0)}}^{\vec{x'}}
  \widetilde \psi_{\vec{w'}}^{\,\vy_{1..n}^{(1)}} \widetilde \phi_{\vec{x'}}^{\,\vz_{n+1..n+n'}^{(1)}} \\
  & =
  \psi_{\vx_{1..m}^{(0)}}^{\vec{w'}} \widetilde \psi_{\vec{w'}}^{\,\vy_{1..n}^{(1)}}
  \phi_{\vu_{m+1..m+m'}^{(0)}}^{\vec{x'}} \widetilde \phi_{\vec{x'}}^{\,\vz_{n+1..n+n'}^{(1)}} \\
  & =
  (\widetilde \psi \circ \psi) \otimes (\widetilde \phi \circ \phi)
\end{align*}
Verification of the symmetry and trace axioms is a routine application of the ATS axioms.
\end{proof}

\section{The free ATS and the free traced SMC}

In this section, we will show that the free abstract tensor system over an alphabet induces a free traced symmetric monoidal category. We assume for the remainder of this section that tensor symbols in an alphabet $\mathcal A$ are canonically labelled. That is, they are of the form $\psi_{\vx^\ii}^{\vy^\oo} \in \mathcal A$. As the labels have no \textit{semantic} content, we can always replace an arbitrary alphabet with a canonically labelled one.

Also note that canonically labelled alphabets and monoidal signatures are essentially the same thing. Let $\textrm{Sig}(\mathcal A)$ be the monoidal signature with morphisms $\psi_{\vx^\ii}^{\vy^\oo} \in \mathcal A$ and the $\dom$ and $\cod$ maps defined by:
\[
\dom(\psi_{\vx^\ii}^{\vy^\oo}) = \vec X \qquad\qquad
\cod(\psi_{\vx^\ii}^{\vy^\oo}) = \vec Y
\]
For any signature $S = (\mathcal O, \mathcal M, \dom, \cod)$, it is always possible to define an alphabet $\mathcal A$ such that $S = \catSig$. Thus, we will often use the designation \catSig to refer to an \textit{arbitrary} monoidal signature.

%with tensor symbols $\psi_{\va}^{\vb}$ for all $\psi \in \mathcal M$ satisfying the above equations, for a suitable choice of labels \va and \vb.

For $\atsFree$ the free ATS over $\mathcal A$, we will show that $\mathbb C[\atsFree]$ is the free traced SMC over \catSig. We will do this by first considering the strict case, where we construct the unique strict traced symmetric monoidal functor $\widetilde v$ that completes the following diagram, for signature homomorphisms $\eta, v$:
\begin{equation}\label{eq:free-triangle}
  \begin{tikzpicture}
    \matrix (m) [cdiag] {
      \catSig & \catFree  \\
        & \mathcal C \\
    };
    \path [arrs]
      (m-1-1) edge [-latex] node {$\eta$} (m-1-2)
      (m-1-1) edge node [swap] {$v$} (m-2-2)
      (m-1-2) edge [dashed] node {$\widetilde v$} (m-2-2);
  \end{tikzpicture}
\end{equation}

Before we get to the bulk of the proof, we introduce some notation. The first thing we introduce is the notion of \textit{labelling} a morphism.

% \begin{definition}\label{def:x-word}
%   For a small, strict monoidal category $\mathcal V$, fix a set of \textit{atomic objects} $O$, such that any object in $\mathcal V$ is a monoidal product of elements of $O$. For an object $X \in \textrm{ob}\mathcal V$, an \textit{$X$-word} is a monoidal product $X_{i_1} \otimes \ldots \otimes X_{i_M} = X$ such that all $i_k$ are distinct and $X_{i_k} \in O$.
% \end{definition}

% We will assume that $O$ contains ``enough copies'' of every atomic object to find $X$-words for every object $X$. Replacing $X$ with an $X$-word is simply the act of binding each position in the monoidal product to a unique index that we can refer to later.

\begin{definition}\label{def:labelling}
  For a set $\mathcal L$ of labels and a function $\mu : \mathcal L \to \textrm{ob}(\mathcal C)$, an object is called $\mu$-labelled if it is equipped with a list \vi such that:
  \[ X = \mu(i_1) \otimes \mu(i_2) \otimes \ldots \otimes \mu(i_n) \]
  A morphism is called $\mu$-labelled if its domain and codomain have $\mu$-labellings for disjoint lists \vi, \vj.
\end{definition}

To simplify notation, we write $\mu$ labelled objects as follows:
\[ X = X_{i_1} \otimes X_{i_2} \otimes \ldots \otimes X_{i_n} \]
where $X_{i_k} = \mu(i_k)$. For a $\mu$-labelled object $(X,\vec i\,)$ and a label $i \in \vi$, $\sigma_{X:i}$ is the (unique) symmetry map that permutes the object $X_i$ to the end of the list and leaves the other objects fixed.
\begin{align*}
  \sigma_{X:x} & = %
\beginpgfgraphicnamed{sigma_i}
\begin{tikzpicture}
	\begin{pgfonlayer}{nodelayer}
		\node [anchor=north, style=none, font=\footnotesize, yshift=0.5 mm] (0) at (-1.5, -1) {$X_{i_1}$};
		\node [anchor=north, style=none, font=\footnotesize, yshift=0.5 mm] (1) at (0, -1) {$X_{i}$};
		\node [anchor=north, style=none, font=\footnotesize, yshift=0.5 mm] (2) at (1.5, -1) {$X_{i_M}$};
		\node [style=none] (3) at (-1.5, -0.75) {};
		\node [style=none] (4) at (-0.5, -0.75) {};
		\node [style=none] (5) at (0, -0.75) {};
		\node [style=none] (6) at (0.5, -0.75) {};
		\node [style=none] (7) at (1.5, -0.75) {};
		\node [style=none] (8) at (1, -0.5) {...};
		\node [style=none] (9) at (-1.5, -0.25) {};
		\node [style=none] (10) at (-0.5, -0.25) {};
		\node [style=none] (11) at (0, -0.25) {};
		\node [style=none] (12) at (0.5, -0.25) {};
		\node [style=none] (13) at (1.5, -0.25) {};
		\node [style=none] (14) at (-1, 0) {...};
		\node [style=none] (15) at (-1.5, 0.25) {};
		\node [style=none] (16) at (-0.5, 0.25) {};
		\node [style=none] (17) at (0.5, 0.25) {};
		\node [style=none] (18) at (1.5, 0.25) {};
		\node [style=none] (19) at (2, 0.25) {};
		\node [style=none] (20) at (1, 0.5) {...};
		\node [style=none] (21) at (-1.5, 0.75) {};
		\node [style=none] (22) at (-0.5, 0.75) {};
		\node [style=none] (23) at (0.5, 0.75) {};
		\node [style=none] (24) at (1.5, 0.75) {};
		\node [style=none] (25) at (2, 0.75) {};
		\node [anchor=north, style=none, font=\footnotesize, yshift=2 mm] (26) at (-1.5, 1) {$X_{i_1}$};
		\node [anchor=north, style=none, font=\footnotesize, yshift=2 mm] (27) at (1.5, 1) {$X_{i_M}$};
		\node [anchor=north, style=none, font=\footnotesize, yshift=2 mm] (28) at (2, 1) {$X_{i}$};
	\end{pgfonlayer}
	\begin{pgfonlayer}{edgelayer}
		\draw [style=diredge] (15.center) to (21.center);
		\draw [style=diredge] (19.center) to (25.center);
		\draw (13.center) to (18.center);
		\draw [style=diredge] (16.center) to (22.center);
		\draw [style=diredge] (18.center) to (24.center);
		\draw (7.center) to (13.center);
		\draw (10.center) to (16.center);
		\draw (5.center) to (11.center);
		\draw (4.center) to (10.center);
		\draw (3.center) to (9.center);
		\draw [style=diredge] (17.center) to (23.center);
		\draw (12.center) to (17.center);
		\draw (11.center) to (19.center);
		\draw (6.center) to (12.center);
		\draw (9.center) to (15.center);
	\end{pgfonlayer}
\end{tikzpicture}}
\endpgfgraphicnamed
\end{align*}

In any traced SMC, we can define a contraction operator $C_i^j(-)$ which ``traces together'' the $i$-th input with the $j$-th output on a labelled morphism.

\begin{definition}
  Let $f : X_{i_1} \otimes \ldots \otimes X_{i_M} \rightarrow  Y_{j_1} \otimes \ldots \otimes Y_{j_N}$ be a labelled morphism in a traced symmetric monoidal category such that for labels $i \in \{ i_1,\ldots,i_M \}$ and $j \in \{ j_1,\ldots,j_N \}$, $X_i = Y_j$. Then we define the \textit{trace contraction} $C_i^j(f)$ as follows:
  \begin{align*}
    C_i^j(f) & := \Tr^{X_i = Y_j}(\sigma_{Y:j} \circ f \circ \sigma_{X:i}^{-1})
  \end{align*}
\end{definition}

%  \\ &  = \ \ \tikzfig{contraction_def}

Note that a contraction of a labelled morphism yields another labelled morphism, by deleting the contracted objects from the label lists. Thus we can contract many times, and the resulting morphism does not depend on the order in which we perform contractions.

\begin{lemma} \rm \label{lem:contractions-commutative}
  Contractions are commutative. For a labelled morphism $f$ distinct indices $i,i'$ and $j,j'$:
  \[ C_{i}^{j}(C_{i'}^{j'}(f)) = C_{i'}^{j'}(C_{i}^{j}(f)) \]
\end{lemma}

\begin{definition}\label{def:disconnected-labelled}
  For a strict traced symmetric monoidal category $\mathcal C$, define a set $M$ of \textit{atomic} morphisms, such that any morphism in $\mathcal C$ can be obtained from those morphisms and the traced symmetric structure. An labelled morphism is called \textit{disconnected} if it is of the form $f = f_1 \otimes \ldots \otimes f_K$, where each $f_k$ is a labelled morphism in $M$:
  \[ f_k : X_{i_{k,1}} \otimes \ldots X_{i_{k,M_k}} \rightarrow Y_{j_{k,1}} \otimes \ldots \otimes Y_{j_{k,N_k}} \]
\end{definition}

\begin{definition}\label{def:cnf}
  Let $f = f_1 \otimes \ldots \otimes f_M$ be a disconnected labelled morphism. For distinct indices $\{i_1,\ldots,i_P\} \subseteq \{i_{1,1},\ldots i_{K,M_K} \}$ and $\{j_1,\ldots,j_P\} \subseteq \{j_{1,1},\ldots j_{K,N_K} \}$, a map $f'$ is said to be in \textit{contraction normal form} (CNF) if:
  \[ f' = C_{i_1}^{j_1}(C_{i_2}^{j_2}(\ldots(C_{i_P}^{j_P}(f))\ldots)) \]
\end{definition}

\begin{definition}\label{def:totally-contracted}
  Let $f$ and $f'$ be given as in Definition \ref{def:cnf}. A component $f_k$ of $f$ is said to be \textit{totally contracted} if the indices of all of its inputs occur in $\{i_1,\ldots,i_P\}$ and the indices of all of its outputs occur in $\{j_1,\ldots,j_P\}$.
\end{definition}

\begin{lemma}\rm \label{lem:reorder-tc}
  For $f$ and $f'$ from Definition \ref{def:cnf}, totally contracted components of $f$ can be re-ordered arbitrarily by relabelling.
\end{lemma}

\begin{lemma}\rm \label{cor:totally-contracted-ids}
  Let $f, f'$ be Defined as in \ref{def:cnf}. If $f_k = 1_{X_{i_{k,1}}} = 1_{Y_{j_{k,1}}}$ is a totally contracted identity map that is not a circle (i.e. it is not contracted to itself), then it can be removed by relabelling.
\end{lemma}

% \begin{proof}
%   From Lemma \ref{lem:reorder-tc}, we can assume the totally contracted identity map is on the far right. Let $M$ be the label of the input and $N$ be the label of the output. By \ref{lem:contractions-commutative}, we can move the two contractions involving the identity map all the way to the inside, so $f'$ is of the form:
%   \[ f' = C(C(\ldots(C_M^j(C_i^N(f'' \otimes 1_{X_M})))\ldots)) \]
  
%   We can reduce the inner map using the definition of $\sigma_{Y:j}$ and the trace axioms.
%   \begin{align*}
%     C_{M}^{j}(C_{i}^{N}(f''\otimes1_{X_{M}}))
%      & =\Tr^{X_{M}=Y_{j}}(\sigma_{Y:j}\circ\Tr^{X_{i}=Y_{N}}((f''\otimes1_{X_{M}})\circ\sigma_{X:i}^{-1}))\\
%      & =\Tr^{X_{M}=Y_{j}}(\sigma_{Y:j}\circ\Tr^{X_{i}=Y_{N}}((f''\otimes1_{X_{M}})\circ(\sigma_{X:i}^{-1}\otimes1_{X_{M}})\circ(1_{X'}\otimes\sigma_{X_{M},X_{i}}^{-1})))\\
%      & =\Tr^{X_{M}=Y_{j}}(\sigma_{Y:j}\circ\Tr^{X_{i}=Y_{N}}(((f''\circ\sigma_{X:i}^{-1})\otimes1_{X_{M}})\circ(1_{X'}\otimes\sigma_{X_{M},X_{i}}^{-1})))\\
%      & =\Tr^{X_{M}=Y_{j}}(\sigma_{Y:j}\circ f''\circ\sigma_{X:i}^{-1}\circ\Tr^{X_{i}=Y_{N}}((1_{X'}\otimes\sigma_{X_{M},X_{i}}^{-1})))\\
%      & =\Tr^{X_{M}=Y_{j}}(\sigma_{Y:j}\circ f''\circ\sigma_{X:i}^{-1}\circ(1_{X'}\otimes\Tr^{X_{i}=Y_{N}}(\sigma_{X_{M},X_{i}}^{-1})))\\
%      & =\Tr^{X_{M}=Y_{j}}(\sigma_{Y:j}\circ f''\circ\sigma_{X:i}^{-1}\circ(1_{X'}\otimes1_{X_{M}=X_{i}}))\\
%      & =\Tr^{X_{i}=Y_{j}}(\sigma_{Y:j}\circ f''\circ\sigma_{X:i}^{-1})=C_{i}^{j}(f'')
%   \end{align*}
% \end{proof}

For full proofs of lemmas \ref{lem:contractions-commutative}, \ref{lem:reorder-tc}, and \ref{cor:totally-contracted-ids}, see \cite{KissingerThesis}. The final ingredient we need for the main theorem is the correspondence between the operations $C_i^j$ and $\mathcal K_i^j$. First, note that labelled morphisms in \catFree are in 1-to-1 correspondence with tensors in $\atsFree$. That is, a morphism $\psi_{\vec x^\ii}^{\vec y^\oo} : \vec X \to \vec Y$ labelled by $(\vec i, \vec j)$ defines the tensor $\psi_{\vec i}^{\vec j}$. By abuse of notation, we will write $\psi_{\vec i}^{\vec j}$ for both the tensor and the corresponding labelled morphism.

% \begin{lemma}
  
% \end{lemma}

% First, note that a morphism $E : \vec X \to \vec Y$ in \catFree labelled by \vi, \vj can naturally be interpreted as an Einstein expression $E^\bullet$ with lower indices \vi and upper indices \vj.
% \[ E^\bullet = E_{x_1^\ii \mapsto i_1, \ldots, x_m^\ii \mapsto i_m}^{y_1^\oo \mapsto j_1, \ldots, y_n^\oo \mapsto j_n} \]

\begin{lemma}\label{lem:c-k-equiv} \rm
  For some fixed objects $\vec X, \vec Y \in \catFree$, a labelled morphism $\psi_{\vec i}^{\vec j} : \vec X \to \vec Y$ in \catFree, and labels $i \in \vec i, j \in \vec j$:
  \[ C_i^j(\psi_{\vec i}^{\vec j}) = \mathcal K_i^j(\psi_{\vec i}^{\vec j}) \]
\end{lemma}

% We need a consistent way to relate sets of labels to lists of labels. To do this fix a total ordering over the labels $\mathcal L$ of $\atsFree$, such that for all $x, y \in \mathcal L$, $x_\ii < y_{(1)}$. For a set of labels $\wx = \{ x_1 < x_2 < \ldots < x_n \}$, let $[\wx] = [x_1,x_2,\ldots,x_n]$ be the associated list.

% Assume that this ordering is also respected by the tensor symbols in $\mathcal A$. That is, all tensor symbols can be written as $\psi_{[\wx]}^{[\wy]}$ for sets $\wx,\wy$. Since we can make a different choice of labels for tensor symbols without changing the semantic content of $\mathcal A$, we can (effectively) do this without loss of generality.

% \begin{itemize}
%   \TODOitem{handle non-strict case}
% \end{itemize}

% \begin{itemize}
%   \TODOitem{cope with reduced Einstein expressions}
% \end{itemize}

\noindent With the help of these lemmas, we are now ready to prove the main theorem.

\begin{theorem}\label{thm:free-tsmc} \rm
  \catFree is the free strict traced symmetric monoidal category over \catSig.
\end{theorem}

\begin{proof}
  Let $\eta$ be the monoidal signature homomorphism that is identity on objects and sends morphisms $\psi_{\vx^\ii}^{\vy^\oo} \in \mathcal A$ to themselves, considered as Einstein expressions with one tensor symbol.
  \[ \psi_{\vx^\ii}^{\vy^\oo} \in \mathcal T(\wx^\ii,\wy^\oo) = \hom_{\catFree}(\vec X, \vec Y) \]
  
  Now, supposing we are given another monoidal signature homomorphism $v$ from $\textrm{Sig}(\mathcal A)$ into a strict traced symmetric monoidal category $\mathcal C$. Our goal is to build a traced symmetric monoidal functor $\widetilde v : \catFree \to \mathcal C$ such that $\widetilde v \eta = v$. On objects:
  \[
  \widetilde v([X_1, X_2, \ldots, X_n]) = v(X_1) \otimes v(X_2) \otimes \ldots \otimes v(X_n)
  \]
  
  Let $|E| : \vec X \to \vec Y$ be morphism in \catFree. In other words, it is an equivalence class of Einstein expressions, up to permutation of tensor symbols, renaming of bound labels, and $\delta$-contraction. Choose some representative $E$ of $|E|$ such that $E$ is of the form:
  \begin{equation}\label{eq:pinned-form}
    E = \delta_{\vx^\ii}^{\vec{x'}} \delta_{\vec{y'}}^{\vy^\oo} \Psi_1 \Psi_2 \ldots \Psi_M
  \end{equation}
  where $\Psi_i$ are tensor symbols (or $\delta$-elements) and $E' = \Psi_1 \Psi_2 \ldots \Psi_M$ is an Einstein expression with upper labels $\vec{x'}$ and lower labels $\vec{y'}$. In other words, $E'$ contains no free labels, and $\vec{x'}$ and $\vec{y'}$ are disjoint. Form an expression $F$ from $E$ by choosing a fresh $j_k$ for each repeated label $i_k$. Reading the repeated indices in $E$ from left to rewrite, the choice of $E$ fixes a unique expression:
  \[
    E = \K_{i_1}^{j_1}(\K_{i_2}^{j_2}(\ldots (\K_{i_N}^{j_N}(F)) \ldots )
  \]
  which, by Lemma~\ref{lem:c-k-equiv} can be expressed:
  \[
    E = C_{i_1}^{j_1}(C_{i_2}^{j_2}(\ldots (C_{i_N}^{j_N}(F) \ldots )
  \]
  up to bound labels $j_k$. Since $F$ contains no repeated labels, 
  \[
    E = C_{i_1}^{j_1}(C_{i_2}^{j_2}(\ldots (C_{i_N}^{j_N}(1_{\vec X} \otimes 1_{\vec Y} \otimes \Psi_1 \otimes \ldots \otimes \Psi_n) \ldots )
  \]
  Then, $\widetilde v$ must respect $v$ and preserve the traced symmetric monoidal structure. In particular, it must preserve $C_i^j$. So, the only possible value for $\widetilde v(\widehat E)$ is:
  \[
    \widetilde v(\widehat E) = C_{i_1}^{j_1}(C_{i_2}^{j_2}(\ldots (C_{i_N}^{j_N}(
      1_{v(\vec X)} \otimes 1_{v(\vec Y)} \otimes v(\Psi_1) \otimes \ldots \otimes v(\Psi_n)) \ldots )
  \]
  where the labelling on the argument is inherited from the labelling on $F$.

  %An an Einstein expression, $E$ is a list of tensor symbols $[\Psi_1, \Psi_2, \ldots, \Psi_n]$, where each $\Psi_i$ is either from $\mathcal A$ (with some relabelling) or a $\delta$-element.

  % If $\Psi_i$ is an element of $\mathcal A$ (ignoring labels), let $\Psi'_i = v(\Psi_i)$. If it is a $\delta$-element, let $\Psi'_i$ be an identity map. Then consider the following morphism in $\mathcal C$:
  % \[
  %   f' := \Psi_1' \otimes \Psi_2' \otimes \ldots \otimes \Psi_n'
  % \]
  % For $\tau$ the typing function on labels, and $v_0$ the object function of $v$, $\mu := v_0 \tau$ is a function from labels into $\textrm{ob}(\mathcal C)$. Now, $\dom(f')$ can be given be given a $\mu$-typing $\vi$ by concatenating the lower labels of all of the $\Psi_k$ and $\cod(f')$ a $\mu$-typing $\vj$ by concatenating the upper labels. Thus $f'$ is $\mu$-labelled. The definition of $\widetilde v$ is then completed by contracting over repeated labels.
  % \[
  %   \widetilde v(\widehat E) = C_{i_1}^{i_1}(C_{i_2}^{i_2}(\ldots(C_{i_m}^{i_m}(f'))\ldots)
  % \]
  
  Then, since all of the non-canonical labels in $E$ are contracted, $\widetilde v(\widehat E)$ is indeed a morphism from $\widetilde v(\vec X)$ to $\widetilde v(\vec Y)$. For this to be well-defined, we need to show it does not depend on the choice of $E$. First, note that the choice of bound labels is irrelevant because the operation $C_i^j$ is defined in terms of (pairs of) \textit{positions} of labels, and does not depend on the choice of labels themselves. Next, consider the form (\ref{eq:pinned-form}) we have fixed for $E$. The order of tensor symbols is fixed except for in $\Psi_1\Psi_2\ldots\Psi_M$. But then, all of the symbols in $\Psi_1\Psi_2\ldots\Psi_M$ must be totally contracted, so by Lemma~\ref{lem:reorder-tc}, the order of the corresponding $v(\Psi_i)$ are irrelevant. Furthermore, $\delta$ expansion or removal will not affect the value of $\widetilde v(\widehat E)$ by corollary~\ref{cor:totally-contracted-ids}. Thus $\widetilde v$ is well-defined.
  
  Next, we show that $\widetilde v$ is a traced symmetric monoidal functor. It follows immediately from the definition that $\widetilde v$ preserves the $C_i^j$ operation:
  \[
  C_i^j(\widetilde v(f)) = \widetilde v(C_i^j(f))
  \]
  where $\widetilde v(f)$ inherits its labelling from $f$. The fact that $\widetilde v$ preserves the monoidal product follows from the definition of $C_i^j$ and the trace axioms. Then, since all of the rest of the traced symmetric monoidal structure can be defined in terms of $C_i^j$ and $\otimes$, $\widetilde v$ must preserve it.
\end{proof}

This suffices to establish that \catFree is the free \textit{strict} traced symmetric monoidal category. The extension to the non-strict case is now routine.

\begin{corollary} \rm
  \catFree is the free traced symmetric monoidal category over \catSig.
\end{corollary}

\begin{proof}
  Theorem~\ref{thm:free-tsmc} establishes the correspondence between valuations and traced symmetric monoidal functors when $\mathcal C$ is strict. The remainder of the proof is similar to that of theorem 1.2 in~\cite{JS}.
\end{proof}

\subsection{The diagrammatic free category}

In \cite{JS}, Joyal and Street defined the free symmetric monoidal category in terms of \textit{anchored diagrams} with valuations.

\begin{definition}
  A generalised topological graph is a topological space $G$ and a distinguished finite discrete subspace $G_0$ such that $G - G_0$ is isomorphic to a disjoint union of finitely many copies of the open interval $(0,1)$ or circles $S^1$.
\end{definition}

The points in $G_0$ are called nodes and the components of $G - G_0$ are called wires. An anchored diagram is then a generalised topological graph with three extra pieces of data:
\begin{enumerate}
  \item A choice of orientation for each wire.
  \item For each $n \in G_0$, a total ordering for the input and output wires of $n$, where inputs and outputs are distinguished using the orientation of wires.
  \item A total ordering on the inputs and outputs of $G$ as a whole, i.e. the ends of wires that are not connected to nodes.
\end{enumerate}

A valuation $v$ of $G$ is then an assignment of objects to wires and morphisms to the points in $G_0$. The data $(G,v)$ can be pictured as follows:
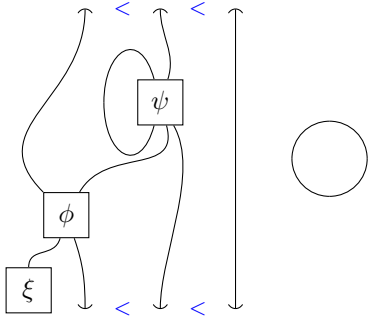
\ctikzfig{anchored_graph}

Let $\textbf{Anchor}(\catSig)$ be the category whose objects are lists of objects from \catSig and whose morphisms are isomorphism classes of anchored diagrams with valuations. Composition is defined by plugging the outputs of one diagram into the inputs of another and monoidal product as diagram juxtaposition. Let $\textbf{Anchor}_p(\catSig)$ be the restriction of morphisms to \textit{progressive} diagrams, i.e. diagrams with no feedback. For details on this construction, see~\cite{JS}.

\begin{theorem}[\cite{JS}] \rm
  $\textbf{Anchor}_p(\catSig)$ is the free symmetric monoidal category on \catSig.
\end{theorem}

If we drop the progressive constraint, $\textbf{Anchor}(\catSig)$ admits a trace operation in the obvious way. So, we can now state the above result for traced symmetric monoidal categories as a corollary to Theorem~\ref{thm:free-tsmc}. This is due to the close relationship between Einstein expressions in diagrams demonstrated in section~\ref{sec:diagrams}.

\begin{corollary} \rm
  $\textbf{Anchor}(\catSig)$ is the free traced symmetric monoidal category on \catSig.
\end{corollary}

\begin{proof}
  We can prove this by demonstrating a (traced symmetric monoidal) isomorphism of categories between $\textbf{Anchor}(\catSig)$ and \catFree. These categories have the same objects, so it suffices to show an isomorphism of hom-sets. By a construction similar to that described in section~\ref{sec:diagrams}, a tensor in \catFree defines a unique anchored diagram, where the total ordering on inputs and outputs is defined using the partial ordering on canonical labels: $x_i \leq y_j \Leftrightarrow i \leq j$. Furthermore, any anchored diagram can be expressed this way.
\end{proof}

\bibliographystyle{akbib}
\bibliography{bibfile}

\end{document}